\documentclass[11pt]{amsart}

\usepackage[margin=1in]{geometry}

\usepackage{amsmath,amsthm,amssymb}
\usepackage[hidelinks]{hyperref}

\newtheorem{theorem}{Theorem}

\newtheorem{corollary}[theorem]{Corollary}
\theoremstyle{remark}
\newtheorem{remark}[theorem]{Remark}

\newcommand{\R}{\mathbb{R}}
\newcommand{\Z}{\mathbb{Z}}
\newcommand{\F}{\mathbb{F}}
\newcommand{\one}{\mathbf{1}}
\newcommand{\E}{\mathbb{E}}
\DeclareMathOperator{\wt}{wt}
\DeclareMathOperator{\Bin}{Bin}
\DeclareMathOperator{\covol}{covol}

\title[Gaussian-Mass Maximality for Construction-A Lattices from Binary Self-Dual Codes]{Majorization and Gaussian-Mass Maximality for Construction-A Lattices from Binary Self-Dual Codes}

\author[Scott Duke Kominers]{Scott Duke Kominers}

\address{Harvard Business School; Department of Economics and Center of
Mathematical Sciences and Applications, Harvard University; and a16z crypto}

\email{kominers@fas.harvard.edu}

\thanks{I used LLMs to assist with computations, analysis, and synthesis in
the preparation of this article, particularly GPT-5.5 Pro and Claude 4.8 Opus
(accessed in part via Poe with the support of Quora, where I am an advisor).
I particularly appreciate a number of helpful comments from Noam~D.\ Elkies
(see in particular Remark~\ref{rmk:other-fields}), as well as a thorough
review from Refine.ink. The problem, methods, and eventual written form are my
own; and of course any errors remain my responsibility. This work was conducted
while I was visiting the Technological Innovation, Entrepreneurship, and
Strategic Management (TIES) Group at the MIT Sloan School of Management; I
greatly appreciate their hospitality.}

\subjclass[2020]{Primary 11H06; Secondary 11H31, 94B05, 94B65}
\keywords{Construction A, self-dual codes, unimodular lattices, theta series, Gaussian mass}

\begin{document}
\maketitle

\begin{abstract}
Regev and Stephens-Davidowitz conjectured that the integer lattice maximizes
Gaussian mass among integral lattices of a given rank. We prove this,
including the equality case, for all unimodular Construction-A lattices arising
from binary self-dual codes. The proof reduces the theta-series inequality to
a sharp majorization statement for codes: if $C$ is a binary self-dual
$[2k,k]$ code, then the half-weight distribution of $C$ is dominated in convex
order by $\Bin(k,1/2)$, which is the corresponding distribution for the
repetition-code model of $\Z^{2k}$. Indeed, after putting $C$ in systematic
form $[I\mid A]$, self-duality gives $AA^T=I$ over $\F_2$, so for a uniformly
random message $a$ the two weights $\wt(a)$ and $\wt(aA)$ have the same
binomial law. The half-weight of the resulting codeword is their average, and
Jensen's inequality then gives convex-order domination. Applied to the convex
test functions that build the theta series, this yields a sum-of-squares
formula for the Gaussian-mass gap; applied to hinge functions, it gives
coefficientwise nonnegativity of the reduced gap polynomial.
\end{abstract}

\section{Introduction}

For a full-rank lattice $\Lambda\subset \R^n$, we write
\[
  \Theta_\Lambda(q)=\sum_{x\in\Lambda} q^{\|x\|^2}
\]
for the \emph{Gaussian mass}, with the real nome $q=e^{-t}\in(0,1)$.
A lattice is \emph{integral} if $\langle x,y\rangle\in\Z$ for all
$x,y\in\Lambda$. Regev and Stephens-Davidowitz conjectured that the integer
lattice maximizes the Gaussian mass among integral lattices of fixed rank,
i.e., that
\begin{equation}
  \Theta_\Lambda(q)\leq \Theta_{\Z^n}(q)\qquad (0<q<1)
\label{eq:conjecture}
\end{equation}
for every integral $\Lambda\subset\R^n$ of rank $n$ \cite{RegevSD_JNT}. This
is the smooth, Gaussian-mass counterpart of the discrete shell bounds of the
reverse-Minkowski program \cite{RegevSD_RevMink,RegevSD_JNT}; the analogous
statement for the Epstein zeta function is known in a range of parameters
\cite{EisenbergRegevSD}, but the theta inequality itself is open. In recent
work~\cite{Kominers_NoGo}, we showed that no Cohn--Elkies-type scalar
Poisson-summation certificate can attain the sharp $\Z^n$ bound in dimensions
$n\geq 8$; even so, in~\cite{Kominers_Rank32}, we were able to establish the
conjecture for every unimodular integral lattice through rank $32$ (and every
even unimodular lattice through rank $40$), using theta-function and
modular-form methods.

Here, we prove the conjectured inequality \eqref{eq:conjecture} for the
unimodular lattices obtained from binary self-dual codes by Construction~A
(Theorem~\ref{thm:square-certificate} and
Corollaries~\ref{cor:gaussian}--\ref{cor:equality}), uniformly across the
family. Moreover, we show that equality occurs only for the direct sum of
length-$2$ repetition codes---equivalently, only when \(L_C\cong\Z^{2k}\).

The result is, at its core, a majorization statement about codes. For a binary
self-dual $[2k,k]$ code~$C$ with associated Construction-A lattice $L_C$, the
inequality $\Theta_{L_C}\leq\Theta_{\Z^{2k}}$ is certified by the stronger
statement that the half-weight distribution of $C$ is dominated in convex
order by $\Bin(k,1/2)$, the corresponding half-weight distribution for the
repetition-code model of $\Z^{2k}$. So $\Z^{2k}$ is the
largest-Gaussian-mass lattice in the family because its associated code---a
direct sum of repetition codes---has the most spread-out half-weight
distribution of any binary self-dual code.

The reason for the domination relationship is a one-line probabilistic
coupling. We write $C$ in systematic form $[I\mid A]$; self-duality forces
$AA^T=I$ over $\F_2$, so $A$ is invertible and the two halves of a uniformly
random codeword, $X=\wt(a)$ and $Y=\wt(aA)$, are each $\Bin(k,1/2)$. The
codeword's half-weight $K=(X+Y)/2$ is therefore an average of two identically
distributed variables, and---by Jensen's inequality---an average is always
dominated in convex order by its marginals. Testing the domination against the
convex functions $t\mapsto z^{2t}$ that assemble the theta series turns it
into a sum of squares:
\[
  \Theta_{\Z^{2k}}(q)-\Theta_{L_C}(q)
  =\theta_3(q^2)^{2k}\cdot\frac12\sum_{a\in\F_2^k}
  \bigl(z^{\wt(a)}-z^{\wt(aA)}\bigr)^2\ \geq 0,
  \qquad z=\frac{\theta_2(q^2)}{\theta_3(q^2)}\in(0,1).
\]

Our sum-of-squares and convex-order argument is not special to \(\F_2\): it
uses only that \(A\) is invertible, so the same certificate covers self-dual
codes over any finite field and Hermitian self-dual codes, with equality
governed by whether $-1$ is a square (respectively, a norm); see
Remark~\ref{rmk:other-fields}.

\subsection*{Relation to prior work.}
That $\Z^n$ should maximize the theta series among unimodular lattices is a
known conjecture, studied through the flatness factor and the smoothing
parameter \cite{BollaufLin}. For the binary self-dual family treated here,
Bollauf and Lin \cite{BollaufLin} reduce the statement to the inequality
$P_C\geq0$ proven in the sequel, and approach it by showing the theta-series
ratio is U-shaped (decreasing then increasing about its symmetry point); this
they verify case-by-case in low dimensions, and unconditionally only on
average over a random self-dual code. Their U-shape route, when available,
additionally locates the global minimum of the ratio at that symmetry point,
which our certificate does not address; conversely, our certificate is
unconditional for the treated family of codes and thus removes the need for
per-code verification.

A closely related question from Gaussian wiretap coding is the Belfiore--Sol\'e
conjecture \cite{BelfioreSole}, that the secrecy function
$\Theta_{\Z^n}/\Theta_\Lambda$ attains its maximum at the symmetry point---the
localization noted above, distinct from the maximality bound itself.
Ernvall-Hyt\"onen~\cite{ErnvallHytonen} established this for all known
extremal even unimodular lattices in dimensions $8$ through $80$, through the
standard expression of the theta series as a polynomial in the Eisenstein
series $E_4$ and the discriminant form $\Delta$; her monotonicity argument
also bounds $\Theta_\Lambda$ by $\Theta_{\Z^n}$ for those specific lattices,
but is neither uniform in the rank nor tied to the self-dual code family. The
results of~\cite{Kominers_Rank32}, meanwhile, are complementary to ours, as
they range over all unimodular lattices but are bounded in rank, whereas the
certificate here is uniform across ranks but specific to unimodular lattices
arising from self-dual codes under Construction A. Our contribution for the
binary self-dual family is an unconditional and elementary proof, valid in all
ranks, together with the equality characterization and the convex-order
interpretation. We stress that the present problem---$\Z^n$ being the
\emph{worst} integral lattice for the theta series---runs opposite to the
much-studied minimization problems for lattice theta and energy, where $E_8$,
the Leech lattice, and the hexagonal lattice are optimal (see \cite{CKMRV}).

\section{Conventions and the Code--Lattice Reduction}

For $0<p<1$ define the Jacobi theta constants
\[
 \theta_2(p)=\sum_{m\in\Z} p^{(m+1/2)^2},\qquad
 \theta_3(p)=\sum_{m\in\Z} p^{m^2},\qquad
 \theta_4(p)=\sum_{m\in\Z} (-1)^m p^{m^2}.
\]
All three are positive on $(0,1)$: this is immediate for $\theta_2$ and
$\theta_3$, and for $\theta_4$ it follows from Jacobi's product identity
$\theta_4(p)=\prod_{m\geq 1}(1-p^{2m})(1-p^{2m-1})^2$. We use the two
duplication identities
\begin{equation}\label{eq:duplication}
  \theta_3(q)^2=\theta_3(q^2)^2+\theta_2(q^2)^2,
  \qquad
  \theta_4(q)^2=\theta_3(q^2)^2-\theta_2(q^2)^2;
\end{equation}
the second, with $\theta_4(q)>0$, gives $0<\theta_2(q^2)<\theta_3(q^2)$, so
that
\begin{equation}\label{eq:zrange}
  z:=\frac{\theta_2(q^2)}{\theta_3(q^2)}\in(0,1)\qquad(0<q<1).
\end{equation}

Let $C\leq \F_2^n$ be a binary linear code, with Hamming weight enumerator
\[
  W_C(X,Y):=\sum_{c\in C}X^{n-\wt(c)}Y^{\wt(c)}
\]
and $W_C(1,z)=\sum_{c\in C}z^{\wt(c)}$. \emph{Construction A} gives
\[
  \Lambda_C=\{x\in\Z^n:x\bmod 2\in C\},
  \qquad L_C=2^{-1/2}\Lambda_C.
\]
The Construction-A lattice $L_C$ is integral iff $C$ is self-orthogonal, and
$\covol(L_C)=2^{n/2-\dim C}$, so $L_C$ is unimodular iff $C$ is self-dual.
Splitting each coordinate by the parity of the corresponding integer gives the
classical identity
\begin{equation}\label{eq:theta-code}
  \Theta_{L_C}(q)=W_C\bigl(\theta_3(q^2),\theta_2(q^2)\bigr),
\end{equation}
since a coordinate $\equiv 0 \pmod 2$ contributes
$\sum_m q^{(2m)^2/2}=\theta_3(q^2)$ and a coordinate $\equiv 1 \pmod 2$
contributes $\sum_m q^{(2m+1)^2/2}=\theta_2(q^2)$ (see, e.g.,
\cite{BroueEnguehard,ConwaySloane,Ebeling}).

Assume now that $C$ is self-dual, so $n=2k$. Since
$\Theta_{\Z^n}(q)=\theta_3(q)^n$, identities \eqref{eq:duplication} and
\eqref{eq:theta-code} give
\begin{equation}\label{eq:gap-reduction}
  \Theta_{\Z^n}(q)-\Theta_{L_C}(q)=\theta_3(q^2)^n\,P_C(z),
\end{equation}
where
\begin{equation}\label{eq:P-expression}
  P_C(z):=(1+z^2)^k-W_C(1,z).
\end{equation}
By \eqref{eq:zrange}, the relevant values of $z$ lie in $(0,1)$; hence the
Gaussian-mass inequality for $L_C$ follows from~\eqref{eq:gap-reduction}
whenever we can show the inequality $P_C(z)\geq 0$ on $[0,1]$ for the
associated self-dual code~$C$.

\section{A Sum-of-Squares Certificate}

\begin{theorem}\label{thm:square-certificate}
Let $C\leq\F_2^{2k}$ be a binary self-dual code. After a coordinate
permutation, choose a systematic generator matrix $G=[I_k\mid A]$. Then
$AA^T=I_k$ over $\F_2$, so $A$ is invertible, and
\begin{equation}\label{eq:square-gap}
  P_C(z)=\frac12\sum_{a\in\F_2^k}
  \left(z^{\wt(a)}-z^{\wt(aA)}\right)^2.
\end{equation}
Moreover $R_C(z):=P_C(z)/\bigl(z^2(1-z^2)^2\bigr)$ is a polynomial in $z^2$
with nonnegative coefficients, given explicitly by
\begin{equation}\label{eq:R-explicit}
  R_C(z)=\frac12\left(
  \sum_{\substack{a\in\F_2^k\\ \wt(a)\ne \wt(aA)}}
  z^{2\min(\wt(a),\wt(aA))-2}
  \left(1+z^2+\cdots+z^{|\wt(a)-\wt(aA)|-2}\right)^2\right).
\end{equation}
In particular, we have $P_C(z)\geq 0$ for all real $z$.
\end{theorem}

The right-hand sides of \eqref{eq:square-gap} and \eqref{eq:R-explicit} depend
on the chosen information set, but $P_C$ and $R_C$ themselves do not: they are
fixed by the weight enumerator $W_C$ through \eqref{eq:P-expression}. The
systematic form merely exhibits one sum-of-squares certificate for this
coordinate-independent gap.

\begin{proof}
Coordinate permutations leave $W_C(1,z)$ and $P_C(z)$ unchanged. A $[2k,k]$
code $C$ has an information set of size \(k\), meaning a set of \(k\)
coordinates on which projection is an isomorphism \(C\to\F_2^k\); after
permuting these coordinates to the front and choosing the corresponding basis
of \(C\), we can write the systematic generator matrix \(G=[I_k\mid A]\). As
$C=C^\perp$ by hypothesis, the rows of $G$ are mutually orthogonal, so
$0=GG^T=I_k+AA^T$ over $\F_2$, giving $AA^T=I_k$; in particular, $A$ is
invertible and $a\mapsto aA$ permutes $\F_2^k$.

The codewords are $c(a)=(a,aA)$, so with $x(a)=\wt(a)$ and
$y(a)=\wt(aA)$,
\[
  W_C(1,z)=\sum_{a\in\F_2^k}z^{x(a)+y(a)}.
\]
Since $A$ is invertible, we have
\[
  (1+z^2)^k=\sum_a z^{2x(a)}=\sum_a z^{2y(a)};
\]
substituting this into the expression \eqref{eq:P-expression} for $P_C$ and
rearranging gives
\begin{align*}
  P_C(z)&=\frac12\sum_a\bigl(z^{2x(a)}+z^{2y(a)}\bigr)
    -\sum_a z^{x(a)+y(a)}\\
  &=\frac12\sum_a\bigl(z^{2x(a)}+z^{2y(a)}-2z^{x(a)+y(a)}\bigr)\\
  &=\frac12\sum_a\bigl(z^{x(a)}-z^{y(a)}\bigr)^2,
\end{align*}
which is \eqref{eq:square-gap}.

Since $C$ is self-orthogonal, every codeword has even weight (over $\F_2$,
$c\cdot c=\wt(c)=0$), so $x(a)$ and $y(a)$ share parity. If $x(a)=y(a)$,
then the associated summand in~\eqref{eq:square-gap} vanishes. If
$x(a)\ne y(a)$, then $d=|x(a)-y(a)|$ is a positive even integer, and $a\ne0$,
so $x(a),y(a)\geq1$ by invertibility of $A$. Hence
\[
  \left(z^{x(a)}-z^{y(a)}\right)^2
  =z^{2\min(x(a),y(a))}(1-z^d)^2,
  \qquad
  1-z^d=(1-z^2)(1+z^2+\cdots+z^{d-2}),
\]
with $2\min(x(a),y(a))\geq 2$. Dividing by $z^2(1-z^2)^2$ term-by-term gives
\eqref{eq:R-explicit}, in which every summand has nonnegative coefficients.
\end{proof}

As a quick illustration: For the extended binary Hamming code $H_8$, we have
\[
  W_{H_8}(1,z)=1+14z^4+z^8,
\]
and hence
\[
  P_{H_8}(z)=(1+z^2)^4-(1+14z^4+z^8)
  =4z^2(1-z^2)^2;
\]
dividing by $z^2(1-z^2)^2$ gives $4$, which of course is positive. For the
extended binary Golay code $G_{24}$, meanwhile, we have
\[
  W_{G_{24}}(1,z)=1+759z^8+2576z^{12}+759z^{16}+z^{24}.
\]
Thus, we have
\begin{align*}
  P_{G_{24}}(z)
  &=(1+z^2)^{12}
    -\bigl(1+759z^8+2576z^{12}+759z^{16}+z^{24}\bigr) \\
  &=12z^2+66z^4+220z^6-264z^8+792z^{10}
    -1652z^{12}  \\
  &\qquad +792z^{14}-264z^{16}+220z^{18}
    +66z^{20}+12z^{22};
\end{align*}
dividing by $z^2(1-z^2)^2$ gives
\[
  \frac{P_{G_{24}}(z)}{z^2(1-z^2)^2}
  =
  12+90z^2+388z^4+422z^6+1248z^8
  +422z^{10}
  +388z^{12}+90z^{14}+12z^{16},
\]
which is manifestly positive.

The general Gaussian-mass consequence follows by substituting
\(z=\theta_2(q^2)/\theta_3(q^2)\) in the reduction
\eqref{eq:gap-reduction}.

\begin{corollary}\label{cor:gaussian}
For every binary self-dual code $C\leq\F_2^n$ and every $0<q<1$, we have
$\Theta_{L_C}(q)\leq \Theta_{\Z^n}(q)$; that is, the
Gaussian-mass--maximality conjecture holds for all Construction-A lattices
arising from binary self-dual codes.
\end{corollary}

\begin{proof}
The result follows from combining \eqref{eq:gap-reduction} with
Theorem~\ref{thm:square-certificate}, as $\theta_3(q^2)^n>0$ and
$z\in(0,1)$.
\end{proof}

\begin{corollary}\label{cor:equality}
Let $C\leq\F_2^{2k}$ be binary and self-dual. Then
$\Theta_{L_C}(q)=\Theta_{\Z^{2k}}(q)$ for some, equivalently every,
$q\in(0,1)$ if and only if, after a coordinate permutation,
$C=\{(a,a):a\in\F_2^k\}$ is the direct sum of $k$ length-two repetition codes;
equivalently $L_C\cong\Z^{2k}$. For every other binary self-dual code the
inequality of Corollary~\ref{cor:gaussian} is strict for all $0<q<1$.
\end{corollary}

\begin{proof}
By \eqref{eq:gap-reduction} and \eqref{eq:zrange}, equality at some $q$ forces
$P_C(z)=0$ at a point $z\in(0,1)$, so every square in~\eqref{eq:square-gap}
vanishes: $\wt(a)=\wt(aA)$ for all $a$. Taking $a=e_i$ shows every row of $A$
has weight $1$; as $A$ is invertible its rows are distinct standard basis
vectors, so $A$ is a permutation matrix. Permuting the second block of
coordinates makes $A=I_k$, i.e., $C=\{(a,a)\}$. The resulting lattice is an
orthogonal sum of $k$ copies of
$2^{-1/2}\{(s,t)\in\Z^2:s\equiv t\bmod 2\}$, which has the orthonormal basis
$(e_i\pm f_i)/\sqrt2$; hence, $L_C\cong\Z^{2k}$ and
$\Theta_{L_C}=\Theta_{\Z^{2k}}$ identically. Conversely this $C$ gives
$W_C(1,z)=(1+z^2)^k$, so $P_C\equiv0$. For any other self-dual code some
square in \eqref{eq:square-gap} is nonzero, so $P_C(z)>0$ for all
$z\in(0,1)$.
\end{proof}

\section{Remarks}

\subsection*{Maximality as majorization.}
Our certificate is best viewed as a concrete manifestation of convex order. Let
$a$ be uniform in $\F_2^k$ and set $X=\wt(a)$, $Y=\wt(aA)$, and
$K=(X+Y)/2$; we have both $X\sim\Bin(k,1/2)$ and $Y\sim\Bin(k,1/2)$ since
$A$ is invertible. As $K$ is an average of two identically distributed
variables, Jensen's inequality gives, for every convex $\varphi$,
\[
  \E[\varphi(K)]\leq \tfrac12\E[\varphi(X)]+\tfrac12\E[\varphi(Y)]=\E[\varphi(X)].
\]
That is, the half-weight law of $C$ is dominated in convex order by
$\Bin(k,1/2)$, the half-weight law of $\Z^{2k}$. The convex functions
$\varphi(t)=z^{2t}$ ($0<z<1$) return the Gaussian-mass inequality of
Theorem~\ref{thm:square-certificate}. The hinge functions
$\varphi_s(t)=(s-t)_+$ return nonnegativity of every coefficient of $R_C$.
Concretely, write
\[
  W_C(1,z)=\sum_{w=0}^{2k}N_w z^w,
\]
so that \(N_w=\#\{c\in C:\wt(c)=w\}\). Since all codewords have even weight,
\[
  P_C(\sqrt u)=\sum_{j=0}^k\left(\binom kj-N_{2j}\right)u^j.
\]
Hence, for \(r\geq0\),
\[
  [u^r]\,\frac{P_C(\sqrt u)}{u(1-u)^2}
  =\sum_{j=0}^k\left(\binom kj-N_{2j}\right)(r+2-j)_+.
\]
Equivalently, if \(X\sim\Bin(k,1/2)\) and \(K\) is the half-weight of a
uniformly random codeword of \(C\), then
\[
  [u^r]\,\frac{P_C(\sqrt u)}{u(1-u)^2}
  =
  2^k\bigl(\E[(r+2-X)_+]-\E[(r+2-K)_+]\bigr)\geq0.
\]
Thus the Gaussian-mass inequality, the sum-of-squares structure, and the
coefficientwise positivity of \(R_C\) are different specializations of the same
convex-order domination. The single integrated inequality of
\cite[Thm.~4]{BollaufLin} is one further convex functional; convex order is
the entire family at once.

\subsection*{Gleason invariance is not enough.}

One might hope that coefficientwise nonnegativity of $R_C$ would follow
directly from Gleason's theorem---or, more precisely, from membership of
$W_C$ in the invariant ring of the weight-enumerator group
\cite{Gleason,MacWilliamsSloane,NebeRainsSloane}. But unfortunately, it does
not appear to.

We again write
\[
  W_C(1,z)=\sum_{w=0}^n N_w z^w.
\]
At length \(24\), the formal Type II weight enumerators form a one-parameter
family indexed by \(\alpha\), the formal coefficient playing the role of \(N_4\).
If \(W_\alpha\) denotes that formal weight enumerator and
\[
  R_\alpha(z):=\frac{(1+z^2)^{12}-W_\alpha(1,z)}{z^2(1-z^2)^2},
\]
then the \(z^2\)-coefficient of \(R_\alpha\) is \(90-\alpha\), which is already negative
for \(\alpha>90\)---while these formal weight enumerators keep nonnegative
coefficients up to \(\alpha=189\), so the failure is not an artifact of
inadmissibility. Genuine Type II codes of length \(24\) instead satisfy
\(N_4\leq 66\), the maximum being realized by the code whose Construction-A
lattice is the Niemeier lattice with root system $D_{24}$
\cite{Niemeier,PlessSloane,NagaokaOura,ConwaySloane}. The bound is thus a fact
about realizable codes, and the proof above secures it through the genuine
coupling $(\wt(a),\wt(aA))$, rather than via the enumerator ring.

\subsection*{Palindromy.}

For $C$ doubly even, we have \(N_2=0\). Moreover \(\one\in C\): since every
codeword has even weight, \(\one\) is orthogonal to every codeword, and hence
\(\one\in C^\perp=C\). Bitwise complementation---i.e., adding the all-$1$s
word---gives a weight-reversing involution \(c\mapsto c+\one\) on \(C\), so
\(N_w=N_{n-w}\) and both \(W_C(1,z)\) and \((1+z^2)^{n/2}\) are palindromic of
degree \(n\). For Type II codes of positive length, the leading \(z^n\) terms
cancel and \(N_{n-2}=N_2=0\), so \(R_C(z)\) is palindromic in the sense that
\[
  R_C(z)=z^{n-8}R_C(1/z);
\]
moreover, since the \(z^{n-2}\) coefficient of \(P_C\) equals
\(\binom{k}{k-1}-N_{n-2}=k=n/2>0\), the degree of \(R_C\) is exactly \(n-8\),
with leading and---by palindromy---constant coefficient \(n/2\).

\section{Scope}

Our argument uses only that $A$ is invertible and that the codewords have even
weight, so it covers all binary self-dual codes---singly-even (Type I) as well
as doubly-even (Type II). Thus Corollary~\ref{cor:gaussian} shows the
Gaussian-mass conjecture, with equality (Corollary~\ref{cor:equality}), for
all unimodular Construction-A lattices from binary self-dual codes, uniformly
across the family. However, our approach does not reach unimodular lattices
outside this Construction-A family: every $2^{-1/2}\Lambda_C$ contains the
vectors $\sqrt2\,e_i$ and so has minimal norm at most $2$, excluding for
instance the Leech lattice. Nor does our method address Construction-A
lattices from self-orthogonal codes that are not self-dual (where $A$ is no
longer square), or integral lattices of determinant greater than $1$. We note
that for the \emph{unscaled} lattice $\Lambda_C=C+2\Z^n$, a sublattice of
$\Z^n$, the bound $\Theta_{\Lambda_C}\leq\Theta_{\Z^n}$ is immediate
(cf.\ \cite[Thm.~5]{BollaufLin}); the substance of
Corollary~\ref{cor:gaussian} is that the rescaled unimodular lattice
$L_C=2^{-1/2}\Lambda_C$, which is not contained in $\Z^n$, still does not
exceed it.

Once a systematic representation \(G=[I\mid A]\) has been chosen, the
sum-of-squares identity itself uses only that \(A\) is invertible. The
additional divisibility by \(z^2(1-z^2)^2\), and the resulting coefficientwise
positivity statement in powers of \(z^2\), use the binary self-orthogonality
fact that all codewords have even Hamming weight. Thus the same averaging
identity extends over other finite fields; this carries a Gaussian-mass
reading wherever a Construction-A theta dictionary exists---directly for prime
fields, and through Eisenstein or Gaussian integer lattices in the Hermitian
case. That said, the accompanying parity-based polynomial refinement is
specific to the binary case unless further hypotheses are imposed.

\begin{remark}[observed by Noam~D.\ Elkies]\label{rmk:other-fields}
Let $C\leq\F_q^{2k}$ be self-dual for the standard bilinear form over a finite
field $\F_q$, in systematic form $[I\mid A]$, so that $AA^T=-I$ and $A$ is
invertible. With $W_C(1,z)=\sum_{c\in C}z^{\wt(c)}$, the same averaging
identity as in Theorem~\ref{thm:square-certificate} gives
\[
  (1+(q-1)z^2)^k-W_C(1,z)
  =\frac12\sum_{a\in\F_q^k}\bigl(z^{\wt(a)}-z^{\wt(aA)}\bigr)^2\ \geq 0,
\]
the reference $(1+(q-1)z^2)^k$ being the weight enumerator of the diagonal code
$\bigoplus^k\langle(1,d)\rangle$ with $d^2=-1$ (when an element \(d\in\F_q\)
with \(d^2=-1\) exists). Equality at some (equivalently every) \(z\in(0,1)\)
forces \(\wt(aA)=\wt(a)\) for all \(a\), hence \(A\) is monomial by the
MacWilliams equivalence theorem \cite{MacWilliamsSloane}. The condition
\(AA^T=-I\) then forces each nonzero monomial entry to have square \(-1\).
Thus equality is attainable exactly when \(-1\) is a square in \(\F_q\)---i.e.,
when \(q\) is even or \(q\equiv1\pmod4\).

The identity holds verbatim for Hermitian self-dual codes over \(\F_{q^2}\),
i.e., codes satisfying \(C=C^{\perp_H}\) for the Hermitian form
\(\langle x,y\rangle_H=\sum_i x_i\overline{y_i}\). In systematic form this
gives \(A\bar A^T=-I\), and the reference is \((1+(q^2-1)z^2)^k\); there
equality requires $-1$ to be a norm, which always holds because the norm
$\F_{q^2}^\times\to\F_q^\times$ is surjective, so the bound is attained, for
instance over $\F_4$.
\end{remark}

For formally self-dual codes the half-weight law remains defined, and whether
the convex-order domination persists is open. General unimodular lattices, by
contrast, need not arise from codes and carry no direct half-weight law, so
even identifying the right majorization analogue there is unclear. A positive
answer for formally self-dual codes would settle the corresponding
code-derived Gaussian-mass bound, while a suitable lattice-level majorization
principle could potentially address the full conjecture.

\newcommand{\etalchar}[1]{$^{#1}$}
\providecommand{\bysame}{\leavevmode\hbox to3em{\hrulefill}\thinspace}
\providecommand{\MR}{\relax\ifhmode\unskip\space\fi MR }
\providecommand{\MRhref}[2]{%
  \href{http://www.ams.org/mathscinet-getitem?mr=#1}{#2}
}
\providecommand{\href}[2]{#2}


\begin{thebibliography}{CKM{\etalchar{+}}22}

\bibitem[BE72]{BroueEnguehard}
Michel Brou\'e and Michel Enguehard, \emph{Polyn\^omes des poids de certains
  codes et fonctions th\^eta de certains r\'eseaux}, Annales scientifiques de
  l'\'Ecole Normale Sup\'erieure (4) \textbf{5} (1972), no.~1, 157--181.

\bibitem[BL25]{BollaufLin}
Maiara~F. Bollauf and Hsuan-Yin Lin, \emph{On the maximum flatness factor over
  unimodular lattices}, preprint, arXiv:2403.16932, 2025.

\bibitem[BS10]{BelfioreSole}
Jean-Claude Belfiore and Patrick Sol\'e, \emph{Unimodular lattices for the
  {Gaussian} wiretap channel}, 2010 IEEE Information Theory Workshop (ITW)
  (Dublin, Ireland), 2010.

\bibitem[CKM{\etalchar{+}}22]{CKMRV}
Henry Cohn, Abhinav Kumar, Stephen~D. Miller, Danylo Radchenko, and Maryna
  Viazovska, \emph{Universal optimality of the {$E_8$} and {Leech} lattices and
  interpolation formulas}, Annals of Mathematics \textbf{196} (2022), no.~3,
  983--1082.

\bibitem[CS99]{ConwaySloane}
John~H. Conway and Neil J.~A. Sloane, \emph{{Sphere Packings, Lattices and
  Groups}}, 3rd ed., Grundlehren der mathematischen Wissenschaften, vol. 290,
  Springer, New York, 1999.

\bibitem[Ebe13]{Ebeling}
Wolfgang Ebeling, \emph{{Lattices and Codes: A Course Partially Based on
  Lectures by Friedrich Hirzebruch}}, 3rd ed., Advanced Lectures in
  Mathematics, Springer Spektrum, Wiesbaden, 2013.

\bibitem[EH12]{ErnvallHytonen}
Anne-Maria Ernvall-Hyt\"onen, \emph{On a conjecture by {Belfiore} and {Sol\'e}
  on some lattices}, IEEE Transactions on Information Theory \textbf{58}
  (2012), no.~9, 5950--5955.

\bibitem[ERSD26]{EisenbergRegevSD}
Yael Eisenberg, Oded Regev, and Noah Stephens-Davidowitz, \emph{A tight reverse
  {M}inkowski inequality for the {E}pstein zeta function}, Proc. Amer. Math.
  Soc. (2026), to appear.

\bibitem[Gle71]{Gleason}
Andrew~M. Gleason, \emph{Weight polynomials of self-dual codes and the
  {MacWilliams} identities}, Actes du Congr\`es International des
  Math\'ematiciens (Nice, 1970), vol. Tome 3, Gauthier-Villars, Paris, 1971,
  pp.~211--215.

\bibitem[Kom26a]{Kominers_NoGo}
Scott~Duke Kominers, \emph{Saturation and no-go theorems for scalar {P}oisson
  certificates of {G}aussian mass maximality}, preprint, arXiv:2605.26803,
  2026.

\bibitem[Kom26b]{Kominers_Rank32}
\bysame, \emph{A sharp reverse {M}inkowski inequality for the {G}aussian mass
  of integral unimodular lattices through rank $32$}, preprint,
  arXiv:2606.01347, 2026.

\bibitem[MS77]{MacWilliamsSloane}
Florence~Jessie MacWilliams and Neil J.~A. Sloane, \emph{{The Theory of
  Error-Correcting Codes}}, North-Holland Mathematical Library, vol.~16,
  North-Holland, Amsterdam, 1977.

\bibitem[Nie73]{Niemeier}
Hans-Volker Niemeier, \emph{Definite quadratische {Formen} der {Dimension} 24
  und {Diskriminante} 1}, Journal of Number Theory \textbf{5} (1973), no.~2,
  142--178.

\bibitem[NO24]{NagaokaOura}
Shoyu Nagaoka and Manabu Oura, \emph{Note on the {Type II} codes of length
  $24$}, Kumamoto J. Math \textbf{37} (2024), 1--9.

\bibitem[NRS06]{NebeRainsSloane}
Gabriele Nebe, Eric~M. Rains, and Neil J.~A. Sloane, \emph{{Self-Dual Codes and
  Invariant Theory}}, Algorithms and Computation in Mathematics, vol.~17,
  Springer, Berlin, 2006.

\bibitem[PS75]{PlessSloane}
Vera Pless and Neil J.~A. Sloane, \emph{On the classification and enumeration
  of self-dual codes}, Journal of Combinatorial Theory, Series A \textbf{18}
  (1975), no.~3, 313--335.

\bibitem[RSD24]{RegevSD_RevMink}
Oded Regev and Noah Stephens-Davidowitz, \emph{A reverse {Minkowski} theorem},
  Annals of Mathematics \textbf{199} (2024), no.~1, 1--49.

\bibitem[RSD26]{RegevSD_JNT}
\bysame, \emph{A simple proof of a reverse {Minkowski} theorem for integral
  lattices}, Journal of Number Theory \textbf{279} (2026), 256--266.

\end{thebibliography}
\end{document}